\documentclass[12pt]{article}
\usepackage{color}
\usepackage[centertags]{amsmath}
\usepackage{amsfonts}
\usepackage{amssymb}
\usepackage{amsthm}
\usepackage{amsfonts}
\usepackage{newlfont}
\usepackage{graphicx}
\usepackage{tikz}

\usepackage{color}
\usepackage{graphicx}
\usepackage{pgf,tikz}
\usetikzlibrary{arrows}
\usetikzlibrary{patterns}
\usepackage{pgfplots}
\usepackage{float}

\usetikzlibrary{shapes,arrows,chains,positioning}
\def\bkR{{\rm I\kern-.17em R}}

\theoremstyle{definition}
\newtheorem{definition}{Definition}[section]

\newtheorem{thm}{Theorem}[section]

\newtheorem{example}{Example}[section]
\newtheorem{corollary}{Corollary}
\theoremstyle{remark}

\begin{document}
\title{An operational test for existence of a consistent increasing quasi-concave value function}

\author{Majid Soleimani-damaneh$^{a,}\footnote{Corresponding
 author.~E-mail: soleimani@khayam.ut.ac.ir.},~~$ Latif Pourkarimi$^{b},~~$\\Pekka J. Korhonen$^{c},~~$Jyrki Wallenius$^{c}$}
\date{}

\maketitle  {\scriptsize $^a$ School of Mathematics, Statistics and
Computer Science, College of Science, University of Tehran, Tehran, Iran.\\ \textcolor{blue}{~~~~~~~~~E-mail: soleimani@khayam.ut.ac.ir}}

{\scriptsize $^b$ Department of Mathematics, Faculty of Sciences, Razi University, Kermanshah, Iran.\\ \textcolor{blue}{~~~~~~~~~E-mail: lp\_karimi@yahoo.com}}

{\scriptsize $^c$ Department of Information and Service Management, Aalto University School of Business, Helsinki, Finland.\\ \textcolor{blue}{~~~~~~~~~E-mail: pekka.korhonen@aalto.fi,~jyrki.wallenius@aalto.fi}}

\maketitle

\begin{abstract}
Existence of an increasing quasi-concave value function consistent with given preference information is an important issue in various fields including Economics, Multiple Criteria Decision Making, and Applied Mathematics. In this paper, we establish necessary and sufficient conditions for existence of a value function satisfying aforementioned properties. This leads to an operational, tractable and easy to use test for checking the existence of a desirable value function. In addition to developing the existence test, we construct consistent linear and non-linear desirable value functions. 
\end{abstract}

\section{Introduction}

Assume that $m$ alternatives represented by $p$-vectors
$$\mathbf{x}_1,\mathbf{x}_2,\ldots,\mathbf{x}_m\in\mathbb{R}^p$$
are given. Suppose that there are no duplicate alternatives, i.e., $\mathbf{x}_i\neq \mathbf{x}_j$ for $i\neq j.$ These vectors can represent different objects in different fields. In Economics, these can be some features of a good, or some characteristics of the economy of a country, etc. In Multiple Criteria Decision Making (MCDM), these are called alternatives and a Decision Maker (DM) is going to rank them and/or choose a most preferred one among them. In Support Vector Machine (SVM) approach (in machine learning), these vectors are corresponding to a training dataset.

Set $\Theta:=\{\mathbf{x}_1,\mathbf{x}_2,\ldots,\mathbf{x}_m\}$. Assume that we have the following pairwise judgements on some selected members of $\Theta$:
\begin{equation}\label{1}
\mathbf{x}_j\succ \mathbf{x}_k,~~\forall j=1,2,\ldots,
t,\end{equation}
where $t<m$ and $k\notin\{1,2,\ldots,t\}$. Here, $\mathbf{x}_j\succ \mathbf{x}_k$ means $\mathbf{x}_j$ is preferred to $\mathbf{x}_k$ by Decision Maker (DM) or planner or Policy Designer (PD). The vector $\mathbf{x}_k$ is fixed. Such
relations can be derived from the value judgements provided by the DM, planner or PD. Here, we assume more is better (In optimization language, our problem is maximization).

Several approaches in the literature of the aforementioned fields have been constructed and applied assuming the existence of an increasing quasi-concave value function consistent with preference information given in (\ref{1}). In fact, quasi-concavity has been introduced by economists and has a central role in economic theory \cite{sam}. In this paper, we concentrate on the existence of such a desirable function consistent with preference information given in (\ref{1}). Precisely speaking, the question is whether there exists an increasing quasi-concave value function $f:\mathbb{R}^p\to\mathbb{R}$ such that
\begin{equation}\label{2}
f(\mathbf{x}_j)>f(\mathbf{x}_k),~~\forall j=1,2,\ldots,
t.\end{equation}

\begin{definition}\label{def}
A function $f:\mathbb{R}^p\longrightarrow \mathbb{R}$ is said to be
\begin{itemize}
\item[(i)] quasi-concave if $f(\lambda
\mathbf{x}+(1-\lambda)\mathbf{y})\geq \min\{f(\mathbf{x}),f(\mathbf{y})\}$ for each $\mathbf{x},\mathbf{y}\in \mathbb{R}^p$ and
each $\lambda\in [0,1]$.
\item[(ii)] increasing if $\mathbf{x}\leq \mathbf{y}$ implies $f(\mathbf{x})\leq f(\mathbf{y})$; Hereafter, the vector-inequalities $\geq,~>,~\leq,$ and $<$ are understood componentwise.
\end{itemize}
\end{definition}

The existence of a value function with some pre-specified properties is a very old and important problem; See, e.g., Birkhoff \cite{bir}, Debreu \cite{deb}, Fishburn \cite{fis}, Karsu \cite{kar}, Keeney and Raiffa \cite{kee}, K\"{o}ksalan et al. \cite{kok}, Korhonen et al. \cite{KMW,kor-1,kor-2017,kor-2}, Nasrabadi et al. \cite{nas}, and Zionts et al. \cite{zio}.  To the best of our knowledge, except for Korhonen et al. \cite{kor-2017}, there is no study about the existence of quasi-concave order-preserving value functions. In the current work, we provide a necessary and sufficient condition for existence of an increasing quasi-concave value function satisfying (\ref{2}). This leads to an operational and easy to use test for checking the consistency of preference information (\ref{1}) with quasi-concavity. In addition to the existence results, we construct consistent linear and non-linear desirable value functions. Section 2 provides some preliminaries. Section 3 is devoted to the main results. Some concluding remarks are given in Section 4.

\section{Some Preliminaries}

As mentioned in the preceding section, in our paper, the vector-inequalities $\geq,~>,~\leq,$ and $<$ are defined componentwise. That is, $\mathbf{x}_1\geq \mathbf{x}_2$ stands for
$x_{i1}\geq x_{i2}$ for each $i=1,2,\ldots,p;$ and analogously $\mathbf{x}_1>\mathbf{x}_2$ means
$x_{i1}>x_{i2}$ for each $i=1,2,\ldots,p$. The vectors are denoted by bolded letters, while non-bold letters are used for scalars. The zero vector is denoted by $\mathbf{0}$ without mentioning its dimension. The non-negative orthant in $\mathbb{R}^p$ is $$\mathbb{R}^p_+:=\{\mathbf{x}\in \mathbb{R}^p:\mathbf{x}\geq \mathbf{0}\}.$$ The positive orthant is  $$\mathbb{R}^p_{++}:=\{\mathbf{x}\in \mathbb{R}^p:\mathbf{x}>\mathbf{0}\}.$$

The norm used in the whole paper is the Euclidean norm.

\begin{definition}
Given a nonempty set $\Omega\subset \mathbb{R}^p$, the distance function corresponding to this set is a function $dist(\cdot;\Omega):\mathbb{R}^p\rightarrow \mathbb{R}$ defined by $$dist(\mathbf{x};\Omega):=\displaystyle\inf_{\omega\in\Omega} \|\mathbf{x}-\mathbf{\omega}\|,~~\mathbf{x}\in\mathbb{R}^p.$$
\end{definition}

It can be seen that $dist(\mathbf{x};\Omega)\geq 0$ for each $\mathbf{x}\in\mathbb{R}^p$; and $dist(\mathbf{x};\Omega)=0$ if and only if $\mathbf{x}\in cl\,\Omega.$ The notation $cl\,\Omega$ denotes the closure of $\Omega$.

\begin{definition}
A function $f:\mathbb{R}^p\longrightarrow \mathbb{R}$ is said to be
\begin{itemize}
\item[(i)] Lipschitz if there exists some $L>0 $ such that $$|f(\mathbf{x})-f(\mathbf{y})|\leq L\|\mathbf{x}-\mathbf{y}\|,~~\forall \mathbf{x},\mathbf{y}\in\mathbb{R}^n.$$
\item[(ii)] concave if $f(\lambda
\mathbf{x}+(1-\lambda)\mathbf{y})\geq \lambda f(\mathbf{x})+(1-\lambda)f(\mathbf{y})$ for each $\mathbf{x},\mathbf{y}\in \mathbb{R}^p$ and
each $\lambda\in [0,1]$.
\end{itemize}
\end{definition}

It can be seen that each concave function is quasi-concave, and each Lipschitz function is continuous.

\section{Main Results}
\subsection{A constructive Approach}

Set
\begin{equation}\label{011new}
D:=\{\mathbf{z}\in
\mathbb{R}^p:~\mathbf{z}=\displaystyle\sum_{j=1}^{t}\lambda_j(\mathbf{x}_j-\mathbf{x}_k),~\lambda_j\geq
0;~j=1,2,\ldots,t\}
\end{equation}
and
\begin{equation}\label{01new}
E:=D+\mathbb{R}_+^p.
\end{equation}
Indeed, $D$ is the convex cone generated by $\mathbf{x}_1-\mathbf{x}_k,\mathbf{x}_2-\mathbf{x}_k,\ldots,\mathbf{x}_t-\mathbf{x}_k$, and $E$ is obtained by adding the first orthant to $D$. These two cones, illustrated in Figure 1, play a vital role in the rest of the paper. In Figure 1, $t=p=2$.

The main idea underlying defining $D$ comes from Jensen's inequality for quasi-concave functions which says: If $f:\mathbb{R}^p\to\mathbb{R}$ is quasi-concave, then $$f(\sum_{i=1}^n\alpha_i\mathbf{x}_i)\geq \min\{f(\mathbf{x}_1),f(\mathbf{x}_2),\ldots,f(\mathbf{x}_n)\},$$ for each $n\in \mathbb{N}$, each $\mathbf{x}_1,\mathbf{x}_2,\ldots,\mathbf{x}_n\in\mathbb{R}^p$, and each $\alpha_1,\alpha_2,\ldots,\alpha_n\geq 0$ with $\sum_{i=1}^n\alpha_i=1$.

As in our setting more is better, the members of $E$ which are located outside of $D$ are at least as good as those in $D$. Mathematically speaking, as value function $f$ is increasing, we have $f(\mathbf{x})\geq f(\mathbf{x}')$ for each $\mathbf{x}\in E\backslash D$ and each $\mathbf{x}'\in D.$

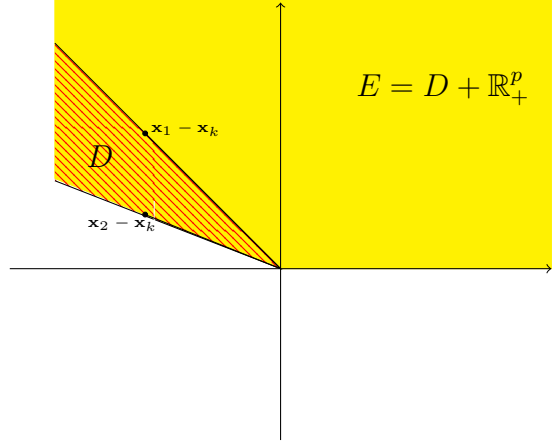
\begin{figure}[H]\label{Fig11}
\hspace{3cm}
\begin{tikzpicture}[scale=0.60,domain=0:3]


\coordinate (A) at (-6,2);
\coordinate (B) at (-6,6);
\coordinate (C) at (5,6);
\coordinate (D) at (-1,0);
\coordinate (E) at (5,0);


\draw[fill=red,yellow]  (A)-- coordinate[pos=.6] (AB) (B)--(C)
         --coordinate[pos=.45] (BC) (C)--(E)
         --coordinate[pos=.55] (CE) (E)--(D)
         --coordinate[pos=.55] (ED) (D)--(A)
         --coordinate[pos=.55] (DA) (A);
\coordinate (A) at (-6,1.95);
\coordinate (B) at (-6,5);
\coordinate (C) at (-1,0);


\draw[pattern=north west lines, pattern color=red]  (A)-- coordinate[pos=.06, color=white] (AC) (C)--(B)
         --coordinate[pos=.05, color=white] (CB) (B);

          \draw[white]    plot [domain = 0:7, samples = 300](2.35-\x,-4) ;
       \draw[white]    plot [domain = 0:7, samples = 300](-3.8,1.5-\x) ;


       \draw[black]    plot [domain = 1:4, samples = 300](-\x,\x-1) ;
       \draw[black]    plot [domain = 1:4, samples = 300](-\x,0.4*\x-0.4) ;



\draw[color=black] (-3.1,3.1) node {\tiny{$\mathbf{x}_1-\mathbf{x}_k$}};
\draw [fill=black] (-4,3) circle (1.5pt);
\draw[color=black] (-4.5,1) node {\tiny{$\mathbf{x}_2-\mathbf{x}_k$}};
\draw [fill=black] (-4,1.2) circle (1.5pt);

\draw[color=black] (2.6,4) node {$E=D+\mathbb{R}^p_+$};
\draw[color=black] (-5,2.5) node {$D$};
    \draw[->] (-7,0) -- (5,0) node[right] {};
    \draw[->] (-1,-3.8) -- (-1,5.9) node[above] {};
\draw[color=black] (0,-4.9) node {};

\end{tikzpicture}
\caption{Illustration of two cones $D$ and $E$.}
\end{figure}

Now, we define a function, $\psi:\mathbb{R}^p\rightarrow \mathbb{R}$, invoking the distance function corresponding to $E$, as follows:
\begin{equation}\label{valuefunc}
\psi(\mathbf{x}):=dist(\mathbf{x}-\mathbf{x}_k; \mathbb R^p\setminus E)-dist(\mathbf{x}-\mathbf{x}_k; E).
\end{equation}

Without loss of generality, we assume both $E$ and $\mathbb R^p\setminus E$ are nonempty.

The following example is provided to clarify Eq. (\ref{valuefunc}). Indeed, the function defined in (\ref{valuefunc}) calculates the closeness and remoteness of $\mathbf{x}$ to $\mathbf{x}_k+E$. When $\mathbf{x}\in \mathbf{x}_k+E$, the value of $\psi(\mathbf{x})$ is nonnegative, while it is negative otherwise. For the members of $\mathbf{x}_k+E$, the function $\psi(\cdot)$ calculates their distance to the boundary of this set.

\begin{example}\label{ex1}
Assume $p=2,~t=3,~\mathbf{x}_k=\left(\begin{matrix}
1\\
1
\end{matrix}\right)~\mathbf{x}_1=\left(\begin{matrix}
0\\
2
\end{matrix}\right)~\mathbf{x}_2=\left(\begin{matrix}
0\\
1.5
\end{matrix}\right)$, and $\mathbf{x}_3=\left(\begin{matrix}
0\\
3
\end{matrix}\right).$ The given preference information is as $\mathbf{x}_j\succ \mathbf{x}_k,~j=1,2,3$. The cone $E$ for this example is depicted in Figure 2. We examine three values for vector $\mathbf{x}$ as follows.
\begin{itemize}
\item $\mathbf{x}^0=\left(\begin{matrix}
-2\\
-2
\end{matrix}\right):$ From Figure 2, it is seen that $dist(\mathbf{x}^0-\mathbf{x}_k; \mathbb R^p\setminus E)=0$ and $dist(\mathbf{x}^0-\mathbf{x}_k; E)=\sqrt{18}$. So, $\psi(\mathbf{x}^0)=-\sqrt{18}$.

\item $\mathbf{x}^*=\left(\begin{matrix}
3\\
3
\end{matrix}\right):$ Due to Figure 2, $dist(\mathbf{x}^*-\mathbf{x}_k; \mathbb R^p\setminus E)=2$ and $dist(\mathbf{x}^*-\mathbf{x}_k; E)=0$. Thus, $\psi(\mathbf{x}^*)=2$.

\item $\bar{\mathbf{x}}=\left(\begin{matrix}
2\\
1
\end{matrix}\right):$ From Figure 2, it is seen that $dist(\bar{\mathbf{x}}-\mathbf{x}_k; \mathbb R^p\setminus E)=dist(\bar{\mathbf{x}}-\mathbf{x}_k; E)=0$. So, $\psi(\bar{\mathbf{x}})=0$.

\end{itemize}

\end{example}

\begin{figure}[H]\label{Fig11}
\hspace{3cm}
\begin{tikzpicture}[scale=0.60,domain=0:3]

\coordinate (A) at (-6,2);
\coordinate (B) at (-6,6);
\coordinate (C) at (5,6);
\coordinate (D) at (-1,0);
\coordinate (E) at (5,0);

\draw[fill=red,green]  (A)-- coordinate[pos=.6] (AB) (B)--(C)
         --coordinate[pos=.45] (BC) (C)--(E)
         --coordinate[pos=.55] (CE) (E)--(D)
         --coordinate[pos=.55] (ED) (D)--(A)
         --coordinate[pos=.55] (DA) (A);

\coordinate (A) at (-6,1.95);
\coordinate (B) at (-6,5);
\coordinate (C) at (-1,0);



\draw[black]    plot [domain = 1:6, samples = 300](-\x,0.4*\x-0.4) ;

 \draw [fill=black] (0,1) circle (1.5pt);
 \draw[color=black] (0.5,1) node {\tiny{$\mathbf{x}_k$}};

\draw [fill=black] (-1,1.5) circle (1.5pt);
 \draw[color=black] (-0.6,1.5) node {\tiny{$\mathbf{x}_2$}};

\draw [fill=black] (-1,2) circle (1.5pt);
 \draw[color=black] (-0.6,2.1) node {\tiny{$\mathbf{x}_1$}};

\draw [fill=black] (-1,3) circle (1.5pt);
 \draw[color=black] (-0.6,3.1) node {\tiny{$\mathbf{x}_3$}};


\draw [fill=black] (-2,1) circle (1.5pt);
 \draw[color=black] (-2.9,1.3) node {\tiny{$\mathbf{x}_1-\mathbf{x}_k$}};

\draw [fill=black] (-2,.4) circle (1.5pt);
 \draw[color=black] (-3,0.4) node {\tiny{$\mathbf{x}_2-\mathbf{x}_k$}};

\draw [fill=black] (-2,2) circle (1.5pt);
 \draw[color=black] (-2.9,2.1) node {\tiny{$\mathbf{x}_3-\mathbf{x}_k$}};

\draw[color=black] (2.6,4) node {$E$};
    \draw[->] (-7,0) -- (5,0) node[right] {};
    \draw[->] (-1,-3.8) -- (-1,5.9) node[above] {};
\draw[color=black] (0,-4.9) node {};

\draw [fill=black] (-4,-3) circle (1.5pt);
 \draw[color=black] (-5,-3) node {\tiny{$\mathbf{x}^0-\mathbf{x}_k$}};

 \draw [fill=black] (1,2) circle (1.5pt);
 \draw[color=black] (2,2) node {\tiny{$\mathbf{x}^*-\mathbf{x}_k$}};

\draw [fill=black] (0,0) circle (1.5pt);
 \draw[color=black] (0.4,-0.3) node {\tiny{$\bar{\mathbf{x}}-\mathbf{x}_k$}};

\end{tikzpicture}
\caption{Illustration of the value function $\psi(\cdot)$ in Example \ref{ex1}.}
\end{figure}
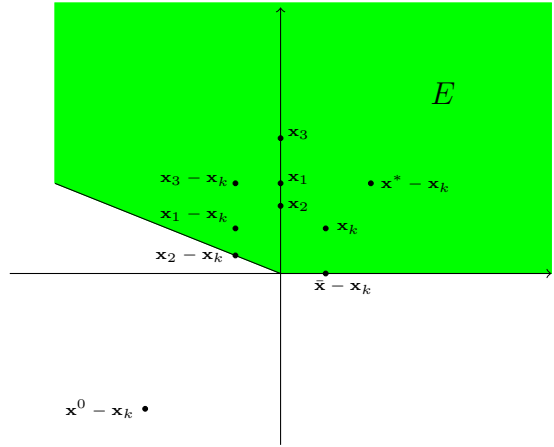

Notice that, we have constructed the function $\psi(\cdot)$ only to show that a continuous quasi-concave increasing value function with some desirable properties exists (See Theorem \ref{main1} and its corollaries below). Here, we neither intend to rank the alternatives by using $\psi(\cdot)$ function nor to study the properties of this function from a ranking standpoint. 

Theorem \ref{main1} below, shows that the function $\psi(\cdot)$ defined above, enjoys some of the important properties that we are looking for. Hereafter, two notations $int$ and $bd$ stand for the interior and boundary of a set, respectively.

\begin{thm}\label{main1} Consider the function $\psi(\cdot)$ defined in (\ref{valuefunc}). The following propositions hold.
\begin{enumerate}
\item[(i)] $\psi(\cdot)$ is a Lipschitz function.
\item[(ii)] For each $\mathbf{x}$,
$$\psi(\mathbf{x})\left\{\begin{array}{ll}
>0, & \textmd{ if } \mathbf{x}\in int\,(\mathbf{x}_k+E)\\
=0, & \textmd{ if } \mathbf{x}\in bd\,(\mathbf{x}_k+E)\\
<0, & \textmd{ if } \mathbf{x}\in int\,\Big(\mathbb{R}^p\setminus (\mathbf{x}_k+E)\Big)
\end{array}\right.$$  
\item[(iii)]  $\psi(\cdot)$ is a concave function.
\item[(iv)] $\psi(\cdot)$ is increasing  w.r.t. $E$, i.e., if $\mathbf{x}_1,\mathbf{x}_2\in\mathbb R^p$ with $\mathbf{x}_2-\mathbf{x}_1\in E$ (resp. $\mathbf{x}_2-\mathbf{x}_1\in int\,E$), then $\psi(\mathbf{x}_1)\leq \psi(\mathbf{x}_2)$ (resp. $\psi(\mathbf{x}_1)< \psi(\mathbf{x}_2)$).
\end{enumerate}
\end{thm}
\begin{proof}
\begin{enumerate}
\item[(i)] Let $\mathbf{x},\mathbf{y}\in\mathbb{R}^p$. As distance function is a Lipschitz function with modulus 1 (see \cite{bec}), we have
$$\begin{array}{ll}
|\psi(\mathbf{x})-\psi(\mathbf{y})|&\leq |dist(\mathbf{x}-\mathbf{x}_k; \mathbb R^p\setminus E)-dist(\mathbf{y}-\mathbf{x}_k; \mathbb R^p\setminus E)|\\
&~~~+|dist(\mathbf{y}-\mathbf{x}_k; E)-dist(\mathbf{x}-\mathbf{x}_k; E)|\\
&\leq \|\mathbf{x}-\mathbf{y}\|+\|\mathbf{x}-\mathbf{y}\|=2\|\mathbf{x}-\mathbf{y}\|,
\end{array}$$
and the desired result is derived.
\item[(ii)] If $\mathbf{x}\in int\,(\mathbf{x}_k+E)$, then $\mathbf{x}-\mathbf{x}_k\in int\,E$ while $\mathbf{x}-\mathbf{x}_k\notin cl\,(\mathbb R^p\setminus E)$ . So, we have $dist(\mathbf{x}-\mathbf{x}_k; E)=0$ and $dist(\mathbf{x}-\mathbf{x}_k; \mathbb R^p\setminus E)>0$, leading to $\psi(\mathbf{x})>0$.

If $\mathbf{x}\in bd\,(\mathbf{x}_k+E)$, then $\mathbf{x}-\mathbf{x}_k\in cl\,E$ and $\mathbf{x}-\mathbf{x}_k\in cl\,(\mathbb R^p\setminus E)$ . So, we have $dist(\mathbf{x}-\mathbf{x}_k; E)=dist(\mathbf{x}-\mathbf{x}_k; \mathbb R^p\setminus E)=0$, leading to $\psi(\mathbf{x})=0$.

The last case is proved analogously.
\item[(iii)] By defining $\varphi(\mathbf{y}):=dist(\mathbf{y}; \mathbb R^p\setminus E)-dist(\mathbf{y}; E)$ for $\mathbf{y}\in\mathbb{R}^p$, the function $\varphi(\cdot)$ is concave due to \cite[Proposition 3.2]{zaf}. On the other hand, $\psi(\mathbf{x})=\varphi(\mathbf{x}-\mathbf{x}_k),~\mathbf{x}\in\mathbb{R}^p$. So, $\psi(\cdot)$ is a concave function.
\item[(iv)] Let $\mathbf{x}_1, \mathbf{x}_2\in\mathbb R^p$ with $\mathbf{x}_2- \mathbf{x}_1\in E$ (resp. $\mathbf{x}_2- \mathbf{x}_1\in int\, E$). Then $\mathbf{x}_2-\mathbf{x}_k-(\mathbf{x}_1-\mathbf{x}_k)\in E$ (resp. $\mathbf{x}_2-\mathbf{x}_k-(\mathbf{x}_1-\mathbf{x}_k)\in int\,E$). Hence, due to \cite[Proposition 3.2]{zaf}, we have $\varphi(\mathbf{x}_2-\mathbf{x}_k)\geq \varphi(\mathbf{x}_1-\mathbf{x}_k),$ (resp. $\varphi(\mathbf{x}_2-\mathbf{x}_k)>\varphi(\mathbf{x}_1-\mathbf{x}_k)$), where $\varphi(\cdot)$ is as defined in the proof of the preceding part. Therefore, $\psi(\mathbf{x}_2)\geq\psi(\mathbf{x}_1)$ (resp. $\psi(\mathbf{x}_2)>\psi(\mathbf{x}_1)$), and the proof is complete.
\end{enumerate}
\end{proof}

The following corollary is a consequence of Theorem \ref{main1}.

\begin{corollary}\label{corNew1} The function $\psi(\cdot)$ defined in (\ref{valuefunc}) is a continuous quasi-concave increasing function satisfying
\begin{equation}\label{1we}
\psi(\mathbf{x}_j)\geq \psi(\mathbf{x}_k),~~\forall j=1,2,\ldots,t.
\end{equation}
\end{corollary}
\begin{proof}
Continuity of $\psi(\cdot)$ comes from Theorem  \ref{main1}(i), because each Lipschitz function is continuous.  Quasi-concavity is derived from Theorem \ref{main1}(iii) because each concave function is quasi-concave. The function $\psi(\cdot)$ is increasing, because considering $\mathbf{x}_1,\mathbf{x}_2\in\mathbb{R}^p$ with $\mathbf{x}_1\leq \mathbf{x}_2$, we have $\mathbf{x}_2-\mathbf{x}_1\in\mathbb{R}^p_+\subseteq E$, and hence, by Theorem \ref{main1}(iv), we get   $\psi(\mathbf{x}_2)\geq\psi(\mathbf{x}_1)$. To prove (\ref{1we}), let $j\in\{1,2,\ldots,t\}$ be arbitrary. Evidently, $\mathbf{x}_j\in \mathbf{x}_k+E$ and so, according to Theorem \ref{main1}(ii), $\psi(\mathbf{x}_j)\geq 0$. On the other hand, $\mathbf{x}_k\in bd(\mathbf{x}_k+E)$, because otherwise we have $\mathbf{x}_k\in int(\mathbf{x}_k+E)$ which implies $0\in int\,E$, and then $E=\mathbb{R}^p$. This contradicts the nonemptiness of $\mathbb{R}^p\setminus E$. Therefore, $\mathbf{x}_k\in bd(\mathbf{x}_k+E)$. Thus, by Theorem \ref{main1}(ii), we get $\psi(\mathbf{x}_k)=0$. This implies $\psi(\mathbf{x}_j)\geq \psi(\mathbf{x}_k)$, and the proof is complete.
\end{proof}

Corollary \ref{corNew1} shows that the function $\psi(\cdot)$ defined in (\ref{valuefunc}) enjoys all but one of the properties that we are looking for. In fact, $\psi(\cdot)$ is a continuous quasi-concave increasing value function fulfilling (\ref{1we}).  If we are satisfied with (\ref{1we}), i.e., we consider (\ref{1we}) as a suitable relation for representing the pairwise judgements (\ref{1}), the value function $\psi(\cdot)$ is desirable.
Notice that, in Corollary \ref{corNew1}, we proved the existence of a continuous increasing quasi-concave value function satisfying (\ref{1we}), without any assumption. If one insists to have a continuous increasing quasi-concave value function satisfying a strict version of (\ref{1we}), as
\begin{equation}\label{2we}
\psi(\mathbf{x}_j)>\psi(\mathbf{x}_k),~~\forall j=1,2,\ldots,t,
\end{equation}
to represent (\ref{1}), we construct such a function by  perturbing the cone $E$ as follows.
We assume that $E$ is pointed (i.e., $E\cap(-E)=\{0\}$), and we will show that (surprisingly!) this assumption is necessary (see Theorem \ref{main5}). For $\epsilon>0$, define
\begin{equation}\label{02new}
D^\epsilon:=\{\mathbf{z}=\displaystyle\sum_{j=1}^{t}\lambda_j (\mathbf{x}_j-\epsilon \mathbf{e}-\mathbf{x}_k),~\lambda_j\geq
0;~j=1,2,\ldots,t\},
\end{equation}
where $\mathbf{e}$ is a $p-$vector whose all components are equal to one. In fact, $D^\epsilon$ is a convex cone generated by $(\mathbf{x}_j-\epsilon \mathbf{e}-\mathbf{x}_k),~j=1,2,\ldots,t$, instead of $(\mathbf{x}_j-\mathbf{x}_k),~j=1,2,\ldots,t$. Set $$E^\epsilon:=D^\epsilon+\mathbb{R}^p_+.$$
Evidently, we have $E^0=E.$ See Figure 3 to have a better insight about two sets $D^\epsilon$ and $E^\epsilon.$ As can be seen from Theorem \ref{main1}, the value function defined in terms of $E$ is zero on the boundary of $x_k+E$. We have defined $E^\epsilon$ such that $\mathbf{x}_j-\mathbf{x}_k\in int\,E^\epsilon,~j=1,2,\ldots,t,$ for each $\epsilon>0$. This leads to $\mathbf{x}_j\succ\mathbf{x}_k,~j=1,2,\ldots,t,$ for the value function defined in terms of some pointed $E^{\bar\epsilon}$; see Theorem \ref{main4}.

As $E$ is pointed, $E^\epsilon$ is pointed for some $\epsilon>0$. Theorem \ref{main2} proves it.

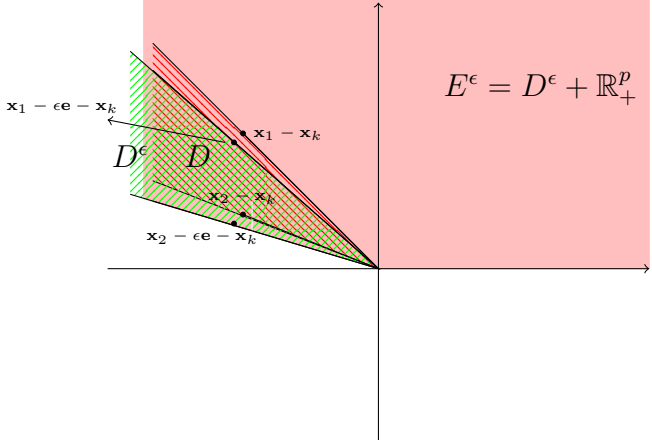
\begin{figure}[H]\label{Fig11}
\begin{tikzpicture}[scale=0.60,domain=0:3]


\coordinate (A) at (-6.2,1.55);
\coordinate (B) at (-6.2,6);
\coordinate (C) at (5,6);
\coordinate (D) at (-1,0);
\coordinate (E) at (5,0);


\draw[fill=red,pink]  (A)-- coordinate[pos=.6] (AB) (B)--(C)
         --coordinate[pos=.45] (BC) (C)--(E)
         --coordinate[pos=.55] (CE) (E)--(D)
         --coordinate[pos=.55] (ED) (D)--(A)
         --coordinate[pos=.55] (DA) (A);
\coordinate (A) at (-6,1.95);
\coordinate (B) at (-6,5);
\coordinate (C) at (-1,0);


\draw[pattern=north west lines, pattern color=red]  (A)-- coordinate[pos=.06, color=white] (AC) (C)--(B)
         --coordinate[pos=.05, color=white] (CB) (B);

   \coordinate (A) at (-6.5,1.65);
\coordinate (B) at (-6.5,4.82);
\coordinate (C) at (-1,0);


\draw[pattern=north east lines, pattern color=green]  (A)-- coordinate[pos=.06, color=white] (AC) (C)--(B)
         --coordinate[pos=.05, color=white] (CB) (B);

          \draw[white]    plot [domain = 0:7, samples = 300](2.35-\x,-4) ;
       \draw[white]    plot [domain = 0:7, samples = 300](-3.8,1.5-\x) ;


       \draw[black]    plot [domain = 1:4, samples = 300](-\x,\x-1) ;
        \draw[black]    plot [domain = 1:6, samples = 300](-\x,0.88*\x-0.88) ;
       \draw[black]    plot [domain = 1:4, samples = 300](-\x,0.4*\x-0.4) ;
       \draw[black]    plot [domain = 1:6, samples = 300](-\x,0.3*\x-0.3) ;



\draw[color=black] (-3,3) node {\tiny{$\mathbf{x}_1-\mathbf{x}_k$}};
\draw [fill=black] (-4,3) circle (1.5pt);
\draw [fill=black] (-4.2,2.8) circle (1.5pt);
\draw[color=black] (-4,1.6) node {\tiny{$\mathbf{x}_2-\mathbf{x}_k$}};
\draw[color=black] (-4.9,0.7) node {\tiny{$\mathbf{x}_2-\epsilon \mathbf{e}-\mathbf{x}_k$}};
\draw[color=black] (-8,3.6) node {\tiny{$\mathbf{x}_1-\epsilon \mathbf{e}-\mathbf{x}_k$}};
 \draw[->] (-4.4,2.8) -- (-7,3.3) node[right] {};
\draw [fill=black] (-4,1.2) circle (1.5pt);
\draw [fill=black] (-4.2,1) circle (1.5pt);

\draw[color=black] (2.6,4) node {$E^\epsilon=D^\epsilon+\mathbb{R}^p_+$};
\draw[color=black] (-5,2.5) node {$D$};
\draw[color=black] (-6.5,2.5) node {$D^\epsilon$};
    \draw[->] (-7,0) -- (5,0) node[right] {};
    \draw[->] (-1,-3.8) -- (-1,5.9) node[above] {};
\draw[color=black] (0,-4.9) node {};

\end{tikzpicture}
\caption{Illustration of two cones $D^\epsilon$ and $E^\epsilon$.~~~~~~~~~~~~~~~~~~~~~~~~~~~~~~~~}
\end{figure}

\begin{thm}\label{main2}
Assume that $E$ is pointed. Then $E^\epsilon$ is pointed for some $\epsilon>0$.
\end{thm}
\begin{proof}
By indirect proof, assume that $E^\epsilon$ is not pointed for any $\epsilon>0$. So, there exists a sequence $\{\mathbf{y}_\epsilon\}_\epsilon\subseteq \mathbb{R}^p\setminus\{0\}$ such that $\mathbf{y}_\epsilon,-\mathbf{y}_\epsilon\in E^\epsilon$ for each $\epsilon>0$. As $E^\epsilon$ is a cone, without loss of generality one may assume $\|\mathbf{y}_\epsilon\|=1$ for each $\epsilon>0$. Hence, without loss of generality (by working with an appropriate subsequence if necessary), one may assume $\mathbf{y}_\epsilon$ converges to some nonzero $\bar{\mathbf{y}}$ as $\epsilon\downarrow 0.$ It can be seen that $\bar{\mathbf{y}},-\bar{\mathbf{y}}\in E$. This contradicts the pointedness of $E$.
\end{proof}

According to Theorem \ref{main2}, $E^{\bar\epsilon}$ is a closed pointed convex cone for some $\bar\epsilon>0.$ As $0<\epsilon_1<\epsilon_2$ implies $E\subseteq E^{\epsilon_1}\subseteq E^{\epsilon_2}$, if $E^{\bar\epsilon}$ is pointed, then $E^\epsilon$ is pointed for each $\epsilon\in[0,\bar\epsilon].$ According to this fact, we sketch Procedure 1 to derive an $\bar\epsilon>0$ for which $E^{\bar\epsilon}$ is pointed. This procedure, which follows a backtracking manner, works based upon the optimal value of the following Linear Programming (LP) problem:
\begin{equation}\label{LP-pointed}
\begin{array}{lll}
z^*_\epsilon:=&\min & \displaystyle\sum_{j=1}^{t+p} r_j\\
                        & s.t.  & (\mathbf{x}_j-\epsilon \mathbf{e}-\mathbf{x}_k)^T\mathbf{d}-s_j+r_j=1,~~j=1,2,\ldots,t,\vspace{2mm}\\
                        &        &d_i-s_{t+i}+r_{t+i}=1,~~i=1,2,\ldots,p,\vspace{2mm}\\
                        &        &s_j,r_j\geq 0,~~j=1,2,\ldots,t+p,\vspace{2mm}\\
                        &&\mathbf{d}\geq 0.
\end{array}
\end{equation}
In LP (\ref{LP-pointed}), $(\mathbf{d},s_1,s_2,\ldots,s_{t+p},r_1,r_2,\ldots,r_{t+p})\in\mathbb{R}^p\times \mathbb{R}^{t+p}\times \mathbb{R}^{t+p}$ is the variable vector. The superscript $\lq\lq ^T"$ stands for transpose.  The quantity $z^*_\epsilon$ denotes the optimal value of the objective function of LP (\ref{LP-pointed}).  This problem is always feasible (set $\mathbf{d}=\mathbf{0},~s_j=0,~r_j=1,~j=1,2,\ldots,t+p$). Furthermore, it always generates optimal solution with $z^*_\epsilon\geq 0.$ Theorem \ref{main3} shows how LP (\ref{LP-pointed}) can be used for checking the pointedness of $E^\epsilon.$\footnote{Checking the pointedness of finitely-generated convex cones may be found in the literature. We proved Theorem \ref{main3} to keep the paper self-contained.}

\begin{thm}\label{main3}
Given $\epsilon>0$, consider LP (\ref{LP-pointed}). The convex cone $E^\epsilon$ is pointed if and only if $z^*_\epsilon=0$.
\end{thm}
\begin{proof}
It is known from the literature that $E^\epsilon$ is pointed if and only if $int\,[E^\epsilon]^+\neq \emptyset$, where $[E^\epsilon]^+$ is the non-negative dual of $E^\epsilon$ defined as
$$[E^\epsilon]^+:=\{\mathbf{d}\in\mathbb{R}^p:~\mathbf{d}^T\mathbf{z}\geq 0,~\forall \mathbf{z}\in E^\epsilon\}.$$
See e.g. \cite[Theorem 2.3]{ber}. We have
$$[E^\epsilon]^+:=\{\mathbf{d}\in\mathbb{R}^p:~~\mathbf{d}\geq \mathbf{0},~~(\mathbf{x}_j-\epsilon \mathbf{e}-\mathbf{x}_k)^T\mathbf{d}\geq 0,~~j=1,2,\ldots,t\}$$
and so
$$int\,[E^\epsilon]^+:=\{\mathbf{d}\in\mathbb{R}^p:~~\mathbf{d}>0,~~ (\mathbf{x}_j-\epsilon \mathbf{e}-\mathbf{x}_k)^T\mathbf{d}>0,~~j=1,2,\ldots,t\}.$$
Therefore, $E^\epsilon$ is pointed if and only if
\begin{equation}\label{290}
\exists  \mathbf{d}\in\mathbb{R}^p~~s.t.~~\mathbf{d}>0,~~(\mathbf{x}_j-\epsilon \mathbf{e}-\mathbf{x}_k)^T \mathbf{d}>0,~~j=1,2,\ldots,t.
\end{equation}
 By normalizing, (\ref{290}) is equivalent to
\begin{equation}\label{291}
\exists \bar{\mathbf{d}}\in\mathbb{R}^p~~s.t.~~\bar{\mathbf{d}}\geq \mathbf{e},~~(\mathbf{x}_j-\epsilon \mathbf{e}-\mathbf{x}_k)^T\bar{\mathbf{d}}\geq 1,~~j=1,2,\ldots,t.
\end{equation}
Furthermore, (\ref{291}) holds if and only if there exists some $\bar{\mathbf{d}}\in\mathbb{R}^p$ such that the vector $$(\mathbf{d}=\bar{\mathbf{d}}, s_j=(\mathbf{x}_j-\epsilon \mathbf{e}-\mathbf{x}_k)^T\bar{\mathbf{d}}-1;~j=1,\ldots,t,~s_j=d_j-1;~j=t+1,\ldots,t+p,~r=0)$$ is feasible for LP (\ref{LP-pointed}) if and only if $z^*_\epsilon=0$. Therefore, $E^\epsilon$ is pointed if and only if $z^*_\epsilon=0$.
\end{proof}

Now, we are ready to present Procedure 1 to derive an $\bar\epsilon>0$ for which $E^{\bar\epsilon}$ is pointed. This backtracking procedure works taking two input parameters $\epsilon_0>0$ (sufficiently small), as an initial value of $\epsilon$, and $\beta\in (0,1)$, as a parameter for contracting the value of $\epsilon$.

$$\begin{tabular}{|l|}\hline
\textbf{Procedure 1:}\\\hline
Inputs.~  $\epsilon_0>0$ (sufficiently small), $\beta\in (0,1)$.\\
Step 1.~  Set $i=0.$\\
Step 2.~  Set $\epsilon=\beta^i\epsilon_0$ and Solve LP (\ref{LP-pointed}).\\
Step 3.~  If $z^*_\epsilon=0$, then stop ($E^\epsilon$ is pointed);\\
~~~~~~~~~~   else\\
~~~~~~~~~~~~~~~set $i=i+1$ and go to Step 2.\\\hline
\end{tabular} $$
$~$\\

Throughout the iterations of Procedure 1, the $\epsilon$ value decreases by multiplying it by
$\beta\in (0,1)$. Recall that if $E^{\bar\epsilon}$ is pointed, then $E^\epsilon$ is pointed for each $\epsilon\in[0,\bar\epsilon].$ Procedure 1 terminates according to this fact and Theorem \ref{main2}.

Now, assume that $\bar\epsilon>0$ derived from Procedure 1 is given. The convex cone $E^{\bar\epsilon}$ is pointed. Similar to $\psi(\cdot)$ defined in (\ref{valuefunc}), we construct another function, $\vartheta:\mathbb{R}^p\rightarrow \mathbb{R}$ as follows:
\begin{equation}\label{valuefunc2}
\vartheta(\mathbf{x}):=dist(\mathbf{x}-\mathbf{x}_k; \mathbb R^p\setminus E^{\bar\epsilon})-dist(\mathbf{x}-\mathbf{x}_k; E^{\bar\epsilon}).
\end{equation}
As you can see from Theorem \ref{main4} below, $\vartheta$ is a continuous increasing quasi-concave value function satisfying strict inequalities $\vartheta(\mathbf{x}_j)>\vartheta(\mathbf{x}_k)$ for all $j=1,2,\ldots,t.$ The proof of this result is similar to that of Theorem \ref{main1} and Corollary \ref{corNew1}, and is hence omitted. The only point which should be taken into account is that here for each $j=1,2,\ldots,t$, we have $\mathbf{x}_j\in int(\mathbf{x}_k+E^{\bar\epsilon})$, and hence $\vartheta(\mathbf{x}_j)>0$. On the other hand, pointedness of $E^{\bar\epsilon}$ implies $\mathbf{x}_k\in bd\,(\mathbf{x}_k+E^{\bar\epsilon})$, and so, $\vartheta(\mathbf{x}_k)=0.$ Therefore,
\begin{equation}\label{3we}
\vartheta(\mathbf{x}_j)>\vartheta(\mathbf{x}_k),~~\forall j=1,2,\ldots,t.
\end{equation}

\begin{thm}\label{main4} Let $\bar\epsilon>0$ be given such that $E^{\bar\epsilon}$ is pointed. Consider the function $\vartheta$ defined in (\ref{valuefunc2}). Then:
\begin{enumerate}
\item[(i)] $\vartheta$ is Lipschitz, and so continuous.
\item[(ii)] For each $\mathbf{x}$,
$$\vartheta(\mathbf{x})\left\{\begin{array}{ll}
>0, & \textmd{ if } \mathbf{x}\in int\,(\mathbf{x}_k+E^{\bar\epsilon})\\
=0, & \textmd{ if } \mathbf{x}\in bd\,(\mathbf{x}_k+E^{\bar\epsilon})\\
<0, & \textmd{ if } \mathbf{x}\in int\,\Big(\mathbb{R}^p\setminus (\mathbf{x}_k+E^{\bar\epsilon})\Big)\\
>0, & \textmd{ if } \mathbf{x}\in \mathbf{x}_k+E,~\mathbf{x}\neq \mathbf{x}_k
\end{array}\right.$$  
\item[(iii)]  $\vartheta$ is concave, and so quasi-concave.
\item[(iv)] $\vartheta$ is increasing  w.r.t. $E^{\bar\epsilon}$.
\item[(v)] $\vartheta$ is increasing in the sense of Definition \ref{def}(ii).
\item[(vi)] $\vartheta(\mathbf{x}_j)>\vartheta(\mathbf{x}_k),~~\forall j=1,2,\ldots,t.$
\end{enumerate}
\end{thm}

Theorem \ref{main4} shows that the function defined in (\ref{valuefunc2}) is a desirable value function that we are looking for. Notice that we proved the existence of such a function only assuming the pointedness of the cone $E$. The interesting point is that, in the following result we show that pointedness of $E$ is necessary for the existence of the value function satisfying the desirable properties.

\begin{thm}\label{main5}
If there exists an increasing quasi-concave value function $\vartheta$ satisfying $\vartheta(\mathbf{x}_j)>\vartheta(\mathbf{x}_k)$ for all $j=1,2,\ldots,t$, then  $E$ is pointed.
\end{thm}
\begin{proof}
By indirect proof, assume that $E$ is not pointed. Then there exists some vector $\mathbf{y}\in\mathbb{R}^p\setminus \{0\}$ such that $\mathbf{y},-\mathbf{y}\in E$. Hence, there exist some $\lambda_j,\mu_j\geq 0,~j=1,2,\ldots,t$ and some $\mathbf{d}_1,\mathbf{d}_2\in\mathbb{R}^p_+$ such that $$\mathbf{y}=\displaystyle\sum_{j=1}^t\lambda_j(\mathbf{x}_j-\mathbf{x}_k)+\mathbf{d}_1,~~~~-\mathbf{y}=\displaystyle\sum_{j=1}^t\mu_j(\mathbf{x}_j-\mathbf{x}_k)+\mathbf{d}_2.$$
If $\lambda_j=\mu_j=0$ for each $j=1,2,\ldots,t$, then $\mathbf{d}_1+\mathbf{d}_2=0$ which implies $\mathbf{d}_1=\mathbf{d}_2=0$, and then $\mathbf{y}=0$. This contradiction shows that $\sum_{j=1}^t\lambda_j+\sum_{j=1}^t\mu_j>0$. Thus, by setting $\theta:=\frac{\sum_{i=1}^t\lambda_i}{\sum_{i=1}^t\lambda_i+\sum_{i=1}^t\mu_i},$ we get $\theta\in [0,1]$ and then
$$\begin{array}{ll}
\vartheta(\mathbf{x}_k)&=\vartheta\bigg(\theta(\mathbf{x}_k+\frac{1}{\sum_{i=1}^t\lambda_i}\mathbf{y})+(1-\theta)(\mathbf{x}_k-\frac{1}{\sum_{i=1}^t\mu_i}\mathbf{y})\bigg)\vspace{2mm}\\
         &\geq \min\bigg\{\vartheta(\mathbf{x}_k+\frac{1}{\sum_{i=1}^t\lambda_i}\mathbf{y}), \vartheta(\mathbf{x}_k-\frac{1}{\sum_{i=1}^t\mu_i}\mathbf{y})\bigg\}~~~\textmd{(by quasi-concavity)}\vspace{2mm}\\
         &=\min\bigg\{\vartheta\bigg(\mathbf{x}_k+\frac{\sum_{j=1}^t\lambda_j(\mathbf{x}_j-\mathbf{x}_k)+\mathbf{d}_1}{\sum_{i=1}^t\lambda_i}\bigg), \vartheta\bigg(\mathbf{x}_k+\frac{\sum_{j=1}^t\mu_j(\mathbf{x}_j-\mathbf{x}_k)+\mathbf{d}_2}{\sum_{i=1}^t\mu_i}\bigg)\bigg\}\vspace{2mm}\\
         &\geq \min\bigg\{\vartheta\bigg(\mathbf{x}_k+\frac{\sum_{j=1}^t\lambda_j(\mathbf{x}_j-\mathbf{x}_k)}{\sum_{i=1}^t\lambda_i}\bigg), \vartheta\bigg(\mathbf{x}_k+\frac{\sum_{j=1}^t\mu_j(\mathbf{x}_j-\mathbf{x}_k)}{\sum_{i=1}^t\mu_i}\bigg)\bigg\}\vspace{-0.5mm}\\
         &\hspace{7cm}\textmd{(by increasing property)}\\
         &=\min\bigg\{\vartheta\bigg(\frac{\sum_{j=1}^t\lambda_j\mathbf{x}_j}{\sum_{i=1}^t\lambda_i}\bigg), \vartheta\bigg(\frac{\sum_{j=1}^t\mu_j\mathbf{x}_j}{\sum_{i=1}^t\mu_i}\bigg)\bigg\}\vspace{2mm}\\
         &\geq \min\{\vartheta(\mathbf{x}_1),\vartheta(\mathbf{x}_2),\ldots,\vartheta(\mathbf{x}_t)\}~~~\textmd{(by quasi-concavity)}\vspace{2mm}\\
         &>\vartheta(\mathbf{x}_k),~~~\textmd{(by assumption of the theorem)}\vspace{-2mm}\\
\end{array}$$
leading to $\vartheta(\mathbf{x}_k)>\vartheta(\mathbf{x}_k)$. This contradiction completes the proof.
\end{proof}

Corollaries \ref{corNew2}-\ref{corNew3} are direct consequences of our discussions so far.

\begin{corollary}\label{corNew2}
There exists an increasing quasi-concave value function $\vartheta$ satisfying $\vartheta(\mathbf{x}_j)>\vartheta(\mathbf{x}_k)$ for all $j=1,2,\ldots,t,$ if and only if $E$ is pointed. This function can be defined via (\ref{valuefunc2}).
\end{corollary}

\begin{corollary}\label{corNew4}
Consider the following LP problem:
\begin{equation}\label{LP-pointed-1}
\begin{array}{lll}
z^*:=&\min & \displaystyle\sum_{j=1}^{t+p} r_j\\
                        & s.t.  & (\mathbf{x}_j-\mathbf{x}_k)^T\mathbf{d}-s_j+r_j=1,~~j=1,2,\ldots,t,\vspace{2mm}\\
                        &        &d_i-s_{t+i}+r_{t+i}=1,~~i=1,2,\ldots,p,\vspace{2mm}\\
                        &        &s_j,r_j\geq 0,~~j=1,2,\ldots,t+p,\vspace{2mm}\\
                        &&\mathbf{d}\geq 0.\\
\end{array}
\end{equation}
Then: there exists an increasing quasi-concave value function $\vartheta$ satisfying $\vartheta(\mathbf{x}_j)>\vartheta(\mathbf{x}_k)$ for all $j=1,2,\ldots,t,$ if and only if $z^*=0.$ This function can be defined via (\ref{valuefunc2}).
\end{corollary}

Corollary \ref{corNew4} provides a tractable, operational, and easy to check test for existence of a value function that we are looking for. Indeed, it can be checked via solving only one linear program, LP (\ref{LP-pointed-1}).\\
LP (\ref{LP-pointed-1}) has been constructed by applying the idea of the proof of Theorem \ref{main3} on $E$. Precisely speaking, $E$ is pointed if and only if $int\,E^+\neq \emptyset$, where $int\,E^+:=\{\mathbf{d}\in\mathbb{R}^p:~~\mathbf{d}>0,~~ (\mathbf{x}_j-\mathbf{x}_k)^T\mathbf{d}>0,~~j=1,2,\ldots,t\}.$
Therefore, $E$ is pointed if and only if there exists some $\bar{\mathbf{d}}\in\mathbb{R}^p$ such that $\bar{\mathbf{d}}\geq \mathbf{e}$ and $(\mathbf{x}_j-\epsilon \mathbf{e}-\mathbf{x}_k)^T\bar{\mathbf{d}}\geq 1,~~j=1,2,\ldots,t.$ This is equivalent to vanishing the optimal value of LP (\ref{LP-pointed-1}); See the proof of Theorem \ref{main3}.

\begin{corollary}\label{corNew3}
If the optimal value of LP (\ref{LP-pointed-1}) is not zero (i.e., $E$ is not pointed), then the given preferences
(relations (\ref{1})) do not come from a DM with an increasing quasi-concave  value function.
\end{corollary}

Examples \ref{ex2} and \ref{ex3} illustrate Corollaries \ref{corNew2}-\ref{corNew3}.

\begin{example}\label{ex2}
Consider Example \ref{ex1} again. The cone $E$, in this example (depicted in Figure 2), is pointed and so (by Corollary \ref{corNew2}) there exists an increasing quasi-concave value function consistent with the given preference information. From Theorem \ref{main5} in the next section, it will be seen that the pointedness of $E$ leads to the existence of a desirable \textit{linear} value function as well. Pointedness of the cones is an important property in multiple criteria decision making. Indeed, if a preference cone is pointed, then the ordering relation induced by this cone is an
antisymmetric relation.\\
LP (\ref{LP-pointed-1}) for this example is as follows:
$$\begin{array}{lll}
z^*:=&\min & \displaystyle r_1+r_2+r_3+r_4+r_5\\
                        & s.t.  & -d_1+d_2-s_1+r_1=1,\vspace{1mm}\\
                        &        &-d_1+0.5d_2-s_2+r_2=1,\vspace{1mm}\\
                        &        &-d_1+2d_2-s_3+r_3=1,\vspace{1mm}\\
                        &        &d_1-s_4+r_4=1,\vspace{1mm}\\
                        &        &d_2-s_5+r_5=1,\vspace{1mm}\\
                        &        &s_1,\ldots,s_5,r_1,\ldots,r_5\geq 0,\vspace{1mm}\\
                        &&d_1,d_2\geq 0.\\
\end{array}$$
The optimal value of this problem is equal to zero, because $$(d_1=1,d_2=4,~s_1=2,~s_2=0,~s_3=6,~s_4=0,~s_5=3,~r_1=r_2=\ldots=d_5=0)$$ is one of its feasible solutions.
\end{example}

\begin{example}\label{ex3}
Assume $p=t=2,~\mathbf{x}_k=\left(\begin{matrix}
1\\
1
\end{matrix}\right)~\mathbf{x}_1=\left(\begin{matrix}
0\\
2
\end{matrix}\right)$, and $\mathbf{x}_3=\left(\begin{matrix}
2\\
0
\end{matrix}\right).$ The given preference information is as $\mathbf{x}_1\succ \mathbf{x}_k$ and $\mathbf{x}_2\succ \mathbf{x}_k$. Here, the cone $E$ is the whole space $\mathbb{R}^2$, and so it is not pointed. Therefore, by Corollary \ref{corNew2}, there does not exist any increasing quasi-concave value function consistent with the given preference information.\\
LP (\ref{LP-pointed-1}) for this example is as follows:
$$\begin{array}{lll}
z^*:=&\min & \displaystyle r_1+r_2+r_3+r_4\\
                        & s.t.  & -d_1+d_2-s_1+r_1=1,\vspace{1mm}\\
                        &        &d_1-d_2-s_2+r_2=1,\vspace{1mm}\\
                        &        &d_1-s_3+r_3=1,\vspace{1mm}\\
                        &        &d_2-s_4+r_4=1,\vspace{1mm}\\
                        &        &s_1,\ldots,s_4,r_1,\ldots,r_4\geq 0,\vspace{1mm}\\
                        &&d_1,d_2\geq 0.\\
\end{array}$$
In each feasible solution of this problem the value of at least one of $r_1$ and $r_2$ is positive. Otherwise, by summing the first two constraints of the problem, we get $s_1+s_2=-1$ which contradicts $s_1,s_2\geq 0.$ Thus, the optimal value of this LP is positive, and so the given preferences cannot be elicited from a DM with an increasing quasi-concave value function (According to Corollary \ref{corNew3}).
\end{example}

\subsection{Linear value function}

In the preceding subsection, we focused on the existence of an increasing quasi-concave value function consistent with preference information (\ref{1}), and showed that it exists if and only if the optimal value of LP (\ref{LP-pointed-1}) is zero. A natural question is, when an increasing \textit{linear} value function consistent with preference information (\ref{1}) exists? As linear value functions have a simpler structure rather than nonlinear ones, this question is important from a practical point of view. The following theorem answers this question.

\begin{thm}\label{main5}
The following three statements are equivalent.
\begin{itemize}
\item[(i)] There exists an increasing quasi-concave value function $\vartheta$ satisfying $\vartheta(\mathbf{x}_j)>\vartheta(\mathbf{x}_k)$ for all $j=1,2,\ldots,t$;
\item[(ii)] There exists an increasing linear value function $\vartheta$ satisfying $\vartheta(\mathbf{x}_j)>\vartheta(\mathbf{x}_k)$ for all $j=1,2,\ldots,t$;
\item[(iii)] The optimal value of LP (\ref{LP-pointed-1}) is zero.
\end{itemize}
\end{thm}
\begin{proof}
The equivalence of (i) and (iii) was proved in the preceding subsection. Furthermore, the implication $(ii)\Longrightarrow (i)$ is trivial. So, we only prove $(i)\Longrightarrow (ii)$. If there exists an increasing quasi-concave value function $\vartheta$ satisfying $\vartheta(\mathbf{x}_j)>\vartheta(\mathbf{x}_k)$ for all $j=1,2,\ldots,t$, then the convex cone $E$ is pointed (Corollary \ref{corNew2}). So, by a manner similar to the proof of Theorem \ref{main3}, there exists some $\mathbf{d}\in\mathbb{R}^p$ such that $\mathbf{d}>0$ and $(\mathbf{x}_j-\mathbf{x}_k)^T\mathbf{d}>0$ for each $j=1,2,\ldots,t.$ Hence, the function $\vartheta(\mathbf{x})=\mathbf{d}^T\mathbf{x}$ is linear and increasing satisfying $\vartheta(\mathbf{x}_j)>\vartheta(\mathbf{x}_k)$ for $j=1,2,\ldots,t$. This completes the proof.
\end{proof}

Theorem \ref{main5} has an interesting message. Given reference information (\ref{1}), when there does not exist a linear increasing value function satisfying (\ref{1}), someone may look for an increasing quasi-concave value function satisfying (\ref{1}). But Theorem \ref{main5} says such a search is futile, because we proved that the existence of a linear increasing value function satisfying (\ref{1}) is equivalent to the existence of a quasi-concave increasing value function satisfying (\ref{1}).

\textbf{\section{Conclusion }}
Existence of an increasing quasi-concave value function consistent with given preference information plays a vital role in various approaches in Economics, Multiple Criteria Decision Making, and Applied Mathematics. Quasi-concave value functions are equivalent with convex to the origin indifference contours; a very fundamental and basic assumption in all economic textbooks. In this paper, we established necessary and sufficient conditions for existence of such a value function. In summary, we showed that the following four statements are equivalent:
\begin{itemize}
\item[(i)] There exists a linear increasing value function satisfying (\ref{1});
\item[(ii)] There exists a quasi-concave increasing value function satisfying (\ref{1});
\item[(iii)] The cone $E$, defined in (\ref{01new}), is pointed;
\item[(iv)]  The optimal value of LP (\ref{LP-pointed-1}) is equal to zero.
\end{itemize}
This full characterization leads to an operational, tractable and easy to use test for checking the consistency of the given preference information with quasi-concavity of the value function. Indeed, according to our results, checking the existence of a desirable value function consistent with given preference information is an easy task and can be done by solving only one LP with $p+t$ constructive constraints, $t+2p$ sign restrictions, and $t+2p$ variables.

\end{document}